\documentclass[11pt,twoside,reqno]{amsart}
\usepackage{amsmath}
\usepackage{amsthm}
\usepackage{amsfonts,amssymb}
\usepackage{bbm}
\usepackage{latexsym}
\usepackage{mathrsfs}
\usepackage[all]{xy}
\usepackage{url}
\usepackage[pdftex]{graphicx}
\usepackage{float}
\usepackage{hyperref}
\usepackage{amssymb,yfonts}
\usepackage{fontenc}
\usepackage{mathdots}
\usepackage[boxsize=23pt]{ytableau}
\usepackage{tikz}
\usepackage{lipsum}
\usepackage{amsthm}
\usepackage{todonotes}
\usepackage[vcentermath]{youngtab}

\usetikzlibrary{cd}

\theoremstyle{plain}
\newtheorem{lema}{Lemma}[section]
\newtheorem{prop}[lema]{Proposition}
\newtheorem{teo}[lema]{Theorem}

\newtheorem{coro}[lema]{Corollary}
\theoremstyle{remark}

\theoremstyle{definition}
\newtheorem{defi}[lema]{Definition}
\newtheorem{ej}[lema]{Example}

\def\QQ{{\mathbb Q}}

\def\a{{\mathrm{a}}}
\def\w{{\mathrm{w}}}

\newcommand{\Si}{\Sigma}

\newcommand{\si}{\sigma}

\begin{document}

\title[The Equivariant Cohomology of Weighted Flag Orbifolds]{The
Equivariant Cohomology of Weighted Flag Orbifolds}
\author[H. Azam]{Haniya Azam}
\author[S. Nazir]{Shaheen Nazir}
\author[M.I. Qureshi]{Muhammad Imran Qureshi}
\address{Haniya Azam, Department of Mathematics, SBASSE\\
Lahore University of Management Sciences (LUMS)\\
Lahore, Pakistan}
\email{haniya.azam@lums.edu.pk}
\address{Shaheen Nazir, Department of Mathematics, SBASSE\\
Lahore University of Management Sciences (LUMS)\\
Lahore, Pakistan}
\email{shaheen.nazir@lums.edu.pk}
\address{Muhammad Imran Qureshi, Department of Mathematics, SBASSE Lahore
University of Management Sciences (LUMS)\\
Lahore, Pakistan and Mathematisches Institut, Universit\"at T\"ubingen,
Germany}
\email{i.qureshi@maths.oxon.org}

\begin{abstract}
We describe the torus-equivariant cohomology of weighted partial flag
orbifolds ${\mathrm{w}}\Sigma$ of type $A$. We establish counterparts of
several results known for the partial flag variety that collectively
constitute what we refer to as ``Schubert Calculus on ${\mathrm{w}}\Sigma$%
''. For the weighed Schubert classes in ${\mathrm{w}}\Sigma$, we give the
Chevalley's formula. In addition, we define the weighted analogue of double
Schubert polynomials and give the corresponding Chevalley--Monk's formula.
\end{abstract}

\maketitle




\section{Introduction}

A flag variety is the quotient of a reductive Lie group $G$ by a unique
parabolic subgroup $P$ (up to conjugation) that is,
\begin{equation*}
\Sigma=G/P.
\end{equation*}
Alternatively, it can be described as a projective subvariety of the
projectivization of some irreducible $G$-representation. The notion of a
weighted flag variety (WFV) ${\mathrm{w}}\Sigma$, is the weighted projective
analogue of the flag variety, introduced by Grojnowski, Corti and Reid \cite%
{C&RweightedGrass}. Any WFV is locally covered by open sets which are
quotients of affine spaces by finite groups, giving it the structure of an
orbifold. Thus we use the terms variety and orbifold interchangeably
throughout the paper. Ever since their introduction, various types of WFVs
have been used to serve as ambient varieties to construct some interesting
classes of polarized orbifolds such as canonical Calabi-Yau 3-folds,
log-terminal Fano 3-folds and canonical 3-folds etc. (see\cite%
{C&RweightedGrass,qs,qs2,qs-ahep,brown2014gorenstein,QJSC,QJGP}). On the
topological side, the equivariant cohomology of weighted Grassmannians has
been computed by Abe and Matsumura in \cite{A&MweightedGrass}.

An important early formal reference on the topology of certain homogeneous
spaces, in particular the flag manifolds, dates back to 1934 by Ehresmann
\cite{ehresmann1934topologie}. Borel in his fundamental work on transformation groups \cite%
{borel2016seminar} defined what is now called the equivariant cohomology of
a space with some group action defined on it.
In his work Borel applied spectral sequence to the topology of Lie groups
and their classifying spaces showing that they degenerate to equivariant
cohomology of homogeneous spaces. In \cite{chang1973topological}, Chang and
Skjelbred proposed the idea of restricting attention to one-dimensional
orbits for calculating equivariant cohomology. Berligne and Vergne \cite%
{berline1983zeros} gave the localization theorem in the context of moment
map, a point of view also adopted by Attiyah and Bott in \cite%
{atiyah1994moment}. An integration of these ideas in terms of the
equivariant cohomology ring of `equivariantly formal spaces' (for instance,
homogeneous spaces) was given by Goresky, Kottwitz and MacPherson in \cite%
{goresky1997equivariant}. They defined the `equivariantly formal spaces' as
spaces whose $G$-equivariant cohomology can be computed by restricting
attention to fixed points and one-dimensional orbits of the maximal torus
inside $G$. The first computation of equivariant cohomology for the complete
flag variety was given by Arabia in \cite{Arabia1986}. Afterwards, many
people used \cite{goresky1997equivariant} to describe equivariant cohomology
of different spaces. The GKM description of the equivariant cohomology ring
of general homogenous spaces has been given by Guillemin, Holm and Zara in
\cite{guillemin2006gkm} and for partial flag varieties of type $A$, by
Tymoczko in \cite{Jul09}.

We compute the rational torus-equivariant cohomology of weighted flag
orbifolds ${\mathrm{w}}\Sigma$ of type $A$, generalizing \cite%
{A&MweightedGrass} to the case of weighted partial flag varieties. We
generalize some known results for partial flag varieties to WFVs, which
includes providing the GKM description of the cohomology ring, a Chevalley's
formula for ${\mathrm{w}}\Sigma$ and the corresponding Chevalley--Monk's
formula in terms of weighted Schubert polynomials defined later.

In \S \ref{S-Rep} we lay the representation theoretic foundations needed in
the rest of the paper. We recall the precise relation between a parabolic
subgroup $P$ and an irreducible representation $V_\chi$ of $G$, where $\chi$
is the highest weight. For a given choice of torus $T$ inside $G$, we
explicitly describe the $T$-invariant basis of the highest weight
representation $V_{\chi}$ using Deyrut's construction of Schur modules. This
construction allows us to explicitly compute the weights of the
representation $V_\chi$.

In \S \ref{S-EqCoh} we recall the Bruhat order on Schubert cells and
describe the open charts and weighted cell decomposition of ${\mathrm{w}}%
\Sigma$ using a $G$-equivariant Pl\"{u}cker type embedding of ${\mathrm{w}}%
\Sigma$. These open charts are all isomorphic to a quotient of a complex
Euclidean space by some finite cyclic group. We give an explicit formula to
compute the ranks of singular rational cohomology groups of these WFVs using
Borel--Moore homology. As a consequence of equivariant formality we show
that the equivariant cohomology ring of ${\mathrm{w}}\Sigma$ admits a basis
over the equivariant cohomology of torus-fixed points.

\begin{teo}
There is an $H^*(BT_{\mathrm{w}})$-module isomorphism
\begin{equation*}
H^ { * } _ {T _ { {\mathrm{w}} } } ( {\mathrm{w}}\Sigma ) \cong H ^ { * } (
B T _ { {\mathrm{w}} } ) \otimes _ { {\mathbb{Q}} } H ^ { * } ( {\mathrm{w}}%
\Sigma ),
\end{equation*}
where $H^*_{T_{{\mathrm{w}}}}(.)$ is the $T_{{\mathrm{w}}}$-equivariant
cohomology and $BT_{{\mathrm{w}}}$ is the classifying space of $T_{{\mathrm{w%
}}}(\cong T)$. In fact, $H^ { * } _ {T _ { {\mathrm{w}} } } ( {\mathrm{w}}%
\Sigma )$ is a free module.
\end{teo}

Apart from giving the module structure we also describe the equivariant
cohomology ring of ${\mathrm{w}}\Sigma$ following Kirwan \cite{Kirwan}. We
use the explicit description of ${\mathrm{w}}\Sigma$ as the quotient of a
compact real symplectic submanifold of the punctured affine cone ${\mathrm{a}%
}\Sigma^\times$ by a Hamiltonian action of the real torus inside ${\mathbb{C}%
}^\times.$

\begin{teo}
There exists a compact real symplectic submanifold $M$ of the punctured
affine cone ${\mathrm{a}}\Sigma^{\times}$ such that
\begin{equation*}
H_{T_{{\mathrm{w}}}}^*({\mathrm{w}}\Sigma)\cong H^*_{S_{T_{{\mathrm{w}}%
}}}(M/S^1) .
\end{equation*}
\end{teo}

In \S \ref{S-GKM} we define Schubert classes both in the weighted flag
orbifold ${\mathrm{w}}\Sigma$ and the punctured affine cone ${\mathrm{a}}%
\Sigma^\times$, using the structure of ${\mathrm{a}}\Sigma^\times$ as a
quasi-projective variety. These classes are defined as pullbacks of
torus-equivariant Schubert classes in the flag variety $\Sigma$. We also
give the combinatorial description, commonly known as GKM description of
equivariant cohomology rings $H^*_K({\mathrm{a}}\Sigma^\times)$ and $H^*_{T_{%
\mathrm{w}}}({\mathrm{w}}\Sigma)$. This is done by describing the image of
the injection
\begin{equation*}
H^*_{T}(\Sigma)\hookrightarrow \bigoplus_{\sigma \in W^P}H^*_T([e_\sigma]),
\end{equation*}
where $e_\sigma$ is a torus-fixed point in \(\Sigma\) and $W^P$ is defined
explicitly in Section \ref{sec:CellDec}.

In \S \ref{S-Pieri} using these GKM descriptions we give Chevalley's formula (Theorem \ref{Th:Chev})
for the weighted flag orbifold ${\mathrm{w}}\Sigma$ by generalizing the
classical formula of Kostant and Kumar \cite{kostant1986nil}. We define
weighted Schubert polynomials ${\mathrm{w}} \mathfrak{S}_{\sigma}(x)$ for a
weighted flag orbifold ${\mathrm{w}}\Sigma$, which generalize the known
double Schubert polynomials defined by Lascoux and Schutzenberger.
The weighted Schubert polynomials correspond to weighted equivariant
cohomology classes of the WFVs. We also give the weighted Chevalley--Monk's
formula (Theorem \ref{Th:Chev-Mon}) in terms of weighted Schubert polynomials.


\section{Irreducible representations of $\mathrm{GL}(n,{\mathbb{C}})$}

\label{S-Rep} 

\subsection{Background}

Let $G=\mathrm{GL}(n,{\mathbb{C}})$, $B$ be the Borel group of upper
triangular matrices, $P$ be a parabolic subgroup of $G$, and $T$ be the
maximal torus of diagonal matrices with associated Lie algebras $\mathfrak{t}%
\subset \mathfrak{b}\subset \mathfrak{p}\subset \mathfrak{gl}_{n}({\mathbb{C}%
})$. Let
\begin{equation*}
\Xi (T)={\mathrm{Hom}}(T,{\mathbb{C}}^{\times })=\langle L_{1},L_{2},\dotsc
,L_{n}\rangle
\end{equation*}%
be the weight lattice. For each $k=1,\dotsc ,n-1$, the $G$-representation $%
\wedge ^{k}{\mathbb{C}}^{n}$ is called fundamental representation of $G$,
having the highest weight
\begin{equation*}
\omega _{k}=\sum_{i=1}^{k}L_{i}\in \Xi (T),\;\;\;\;1\leq k\leq n-1.
\end{equation*}%
Let $\Delta $ be the root system of the Lie algebra $\mathfrak{g}$
corresponding to $G$. Let
\begin{equation*}
\Delta _{s}:=\{\alpha _{i}=L_{i}-L_{i+1},\;1\leq i\leq n-1\}
\end{equation*}%
be the set of simple roots of $\mathfrak{g}$. The Weyl group $W$ of $\Delta $
is generated by reflections in the hyperplane perpendicular to each simple
root $\alpha _{i}$, i.e.,
\begin{equation*}
W=\langle s_{\alpha _{i}}:\alpha _{i}\in \Delta _{s}\rangle \cong S_{n},
\end{equation*}%
where isomorphism with $S_{n}$ is obtained by mapping each simple root $%
\alpha _{i}$ to the simple transposition $s_{i}$. There is a one-one
correspondence between parabolic subgroups and subsets of $\Delta _{s}$ (see
\cite[Section 23.3]{fulton2013representation}). Consider the flag variety
defined by $P$,
\begin{equation*}
\Sigma =G/P:=\{F_{\bullet }:0=V_{0}\subset V_{1}\subset \dotsb \subset
V_{r}\subset V_{r+1}={\mathbb{C}}^{n}|\dim (V_{i})=d_{i}\},
\end{equation*}%
where $P$ is the stabilizer of the $G$-action on each flag in $\Sigma $. The
parabolic subgroup $P$ corresponds to the subset $J=\{\alpha _{i}:i\notin
\{d_{1},\dotsc ,d_{r}\}\}$ of $\Delta _{s}$ and $W_{P}=<s_{i}:\alpha _{i}\in
J>$ denotes the corresponding subgroup of the Weyl group $W$. The dimension
of $\Sigma $ by ~\cite{brion2005lectures} equals
\begin{equation*}
\dim (\Sigma )=\sum_{i=1}^{r}d_{i}(d_{i+1}-d_{i}),\text{ where }d_{r+1}=n.
\end{equation*}%
The flag variety $\Sigma $ is a projective subvariety of ${\mathbb{P}}%
V_{\chi }$, where $V_{\chi }$ is an irreducible $G$-representation with
highest weight
\begin{equation}
\chi =\sum_{i=1}^{r}\omega _{d_{i}}=r\left( L_{1}+\cdots +L_{d_{1}}\right)
+(r-1)\left( L_{d_{1}+1}+\cdots +L_{d_{2}}\right) +\cdots +\left(
L_{d_{r-1}+1}+\cdots +L_{d_{r}}\right) .  \label{chi}
\end{equation}%
Following \cite{fulton1997young}, the Young diagram associated to $\chi $ is
of length $r$ and has type
\begin{equation}
\chi =(\underset{d_{1}}{\underbrace{r,\dotsc ,r}},\underset{d_{2}-d_{1}}{%
\underbrace{r-1,\dotsc ,r-1}},\dotsc ,\underset{d_{r}-d_{r-1}}{\underbrace{%
1,\dotsc ,1}}).  \label{YoungDiagram}
\end{equation}%
That is the first $d_{1}$ rows has $r$-boxes, the next $d_{2}-d_{1}$ rows
has $r-1$ boxes and continuing similarly we get only one box in the last $%
d_{r}-d_{r-1}$ rows.

Let $e_{1} , e_{2} , \ldots, e_{n} $ be the standard basis of ${\mathbb{C}}%
^{n}$, then
\begin{equation*}
v _ { \chi } = \left(e_{1}\wedge \ldots\wedge e_{d_1}\right) \otimes \left(
e_{1}\wedge \ldots \wedge e_{d_2} \right) \otimes \cdots \otimes ( e_{1}
\wedge \ldots \wedge e_{d_r} )
\end{equation*}
in $\wedge^{d_1}({\mathbb{C}} ^ { n }) \otimes \wedge^{d_2} \left( {\mathbb{C%
}} ^{n}\right) \otimes \dotsb \otimes \wedge^{d_r}\left( {\mathbb{C}} ^ { n
} \right) $, is the highest weight vector of
\begin{equation*}
V_ { \chi } = \langle G \cdot v_{ \chi } \rangle
\end{equation*}
by (cf. \cite[Chapter 15]{fulton2013representation}):  We denote the affine
cone over $\Sigma$ by
\begin{equation*}
{\mathrm{a}}\Sigma : = G \cdot {\mathbb{C}} v_{\chi } \subset V _{ \chi } .
\end{equation*}

\subsection{$T $-invariant Basis of $V _ {\protect\chi } $}

To construct an explicit basis of $V _ { \chi } $, we use the \emph{Deyrut's
Construction} of the Schur module \cite{fulton2013representation} with
highest weight $\chi $ which we describe below.

Let ${\mathbb{C}}[X] : = {\mathbb{C}}[ x_{ij} ] _ { 1 \leq i, j \leq n } $
be the polynomial ring in $n^{2} $ variables. Then ${\mathbb{C}}[X] $ is an $%
\mathrm{SL}(n, {\mathbb{C}}) $-module via the following action
\begin{equation*}
A \cdot f ( X ) = f ( A ^ { t } X )\ \text{ for all }\ f \in {\mathbb{C}}%
[X]\ \text{ and } A \in \mathrm{SL}(n, {\mathbb{C}}).
\end{equation*}
A \textit{semi-standard staircase tableau} of length ${r} $ is a filling of
the Young diagram of type \eqref{YoungDiagram} with entries from $\left \{1,
2, \ldots, n \right \}$ which is weakly increasing across each row and
strictly increasing along each column. We denote by $\mathcal{Y} $ the set
of all semi standard staircase tableaux of the type \eqref{YoungDiagram}.

Given a column vector $c= ( c _{1} , c_{2} , \ldots, c_{l } ) ^ {t} $ where $%
1 \leq c_{1} < c_{2} < \dotsb < c_{l} \leq n $ , let
\begin{equation*}
e_{c} : = \det ( x _ { i , c_{j} } ) _ { 1 \leq i , j \leq l } .
\end{equation*}
Let $Y \in \mathcal{Y} $. Let $c$ denote the entries of an arbitrary column
in $Y$ and define
\begin{equation*}
e_{Y} : = \prod _ { c } e_{c} .
\end{equation*}
Then $\langle e_{Y} | Y \in \mathcal{\ Y } \rangle $ is an irreducible
sub-representation of ${\mathbb{C}} [ X ] $ with highest weight $\chi $ and
highest weight vector $e _ { Y _ { 0 } } $ where $Y _ {0} $ is the tableau
in which the $i $-th row is filled with $i $ only (see in the following
Young diagram). \vspace{1em}

\begin{center}
{\scriptsize {%
\begin{ytableau}
 1 & 1 & \cdots &\cdots & \cdots&1 &1 \\

 \vdots & \vdots & \cdots & \cdots &\cdots&\vdots&\vdots \\
 d_1 & d_1 & \cdots &\cdots & \cdots&\cdots &d_1 \\

d_1+1 & \cdots & \cdots & \cdots &\cdots&d_{1}+1 \\
 \vdots & \vdots &\cdots& \cdots & \cdots& \cdots\\

d_2 & \cdots & \cdots & \cdots &\cdots&d_{2} \\
 \vdots & \vdots &\cdots& \cdots & \cdots \\
 d_{r-1} &d_{r-1}  \\

 d_{r-1}+1 \\
 \vdots \\

 d_r

\end{ytableau}} }

The Young tableau $Y_0$ corresponds to the highest weight
\end{center}

\vspace{0.5cm}

Thus, we define $V_{ \chi } $ to be $\langle e_{ Y} \, | \, Y \in \mathcal{Y
} \rangle $. Each $e_{ Y } $ is a weight vector under the action of $T $,
with weight given by
\begin{equation*}
\prod _ { i \in Y } t _ { i } ,
\end{equation*}
where the product runs over all the entries of $Y $. We thus have a $T $%
-invariant basis of $V _ { \chi } $.

\begin{ej}
There are 8 staircase tableaux of shape $(2,1) $ which are written below
along with the corresponding weights.

\begin{equation*}
\begin{array}{cccccccc}
\young(11,2) \, \, & \, \, \young(11,3) \, \, & \, \, \young(12,2) \, \, &
\, \, \young(12,3) \, \, & \, \, \young(13,2) \, \, & \, \, \young(13,3) \,
\, & \, \, \young(22,3) \, \, & \, \, \young(23,3) \\[1em]
t_{1}^{2} t_{2} & t_{1} ^{2} t_{3} & t_{1} t_{2}^{2} & 1 & 1 & t_{1} t_{3}
^{2} & t_{2} ^ { 2} t_{3} & t_{2} t_{3} ^ { 2 }%
\end{array}
\end{equation*}
The weight vectors are all binomials in $x_{ij} $'s. For instance, the
highest weight vector is $e _ { {\scriptsize {\ \young(11,2) } }} = x_{11} (
x_{11} x_{22} - x_{ 1 2} x_{21} ) $.
\end{ej}

\section{Equivariant Cohomology of Weighted Flag orbifolds}

\label{S-EqCoh} The aim of this section is to recall the definition of the
weighted flag varieties, give their cell decomposition and to compute their
cohomology.

It is well-known that the exponential map for a compact Lie group is always
surjective. Thus a one-parameter subgroup in $G$ that is, $\mathbb{C}%
^{\times}\hookrightarrow G\times \mathbb{C}^{\times}$, can be found inside a
maximal torus in $G$. Since all maximal tori are conjugate, so we consider
the maximal torus to be $T$.

Specifying weights on an affine cone over a given manifold is equivalent to
choosing a one-parameter subgroup in $T$ (see \cite{C&RweightedGrass} for a
reference). Let $\rho\in {\mathrm{Hom}}{(\mathbb{C}^{\times},T)}$ be
one-parameter subgroup defined by
\begin{equation*}
\rho ( t ) = \mathrm{diag} ( t ^ { w_{1} } , t ^ { w_{2} } , \ldots , t ^{
w_{ n } } ) \in G ,
\end{equation*}
where $w_1,\ldots, w_n\in {\mathbb{Z}}$. Take $u\in {\mathbb{Z}}_{\geq 0}$,
we use $\rho$ and $u$ to make $V_{\chi}$ into a $G \times {\mathbb{C}} ^ {
\times } $-representation subject to the following action (of the second
component of $G\times \mathbb{C}^{\times}$):
\begin{equation*}
t \cdot v = t ^{ u } \rho ( t) v \, \, \forall \, v \in V_{\chi} , t \in {%
\mathbb{C}} ^ { \times } .
\end{equation*}
We assume that this action has only positive weights, which can be done by
taking $u $ to be sufficiently large \cite{C&RweightedGrass}.

\begin{defi}
\cite{C&RweightedGrass} The \textit{weighted flag variety} associated to the
data $(\rho,u), $ is defined to be
\begin{equation*}
{\mathrm{w}}\Sigma : = {\mathrm{a}}\Sigma\backslash\{0\} /{\mathbb{C}}^ {
\times } ={\mathrm{a}}\Sigma^\times /{\mathbb{C}}^ { \times }\subset {%
\mathrm{w}}{\mathbb{P}} V_\chi (\rho,u)
\end{equation*}
where the quotient is taken by action of ${\mathbb{C}}^{\times}$ on $V_{\chi
} $, as defined above.
\end{defi}

The weighted variety ${\mathrm{w}}\Sigma$ usually contains the quotient
singularities due to the weights of their embeddings in ${\mathrm{w}}{%
\mathbb{P}} V_\chi (\rho,u)$ so we can also refer to them as weighted flag
orbifolds. By definition, the obrifold ${\mathrm{w}}\Sigma $ has a natural
residual action of
\begin{equation*}
T_{{\mathrm{w}}} : = ( T \times {\mathbb{C}} ^ { \times } ) / {\mathbb{C}} ^
{ \times } \cong T
\end{equation*}
such that the action of $T_{{\mathrm{w}}} $ extends to an action on the
ambient space ${\mathrm{w}}{\mathbb{P}} V_\chi (\rho,u)$.

The ${\mathbb{C}} ^ { \times } $-action on $V _ { \chi } $ via $\rho $ has
each $e_{Y} $ as a weight vector with the weight $t ^ { w _ { Y } } $ for $t
\in {\mathbb{C}} ^ { \times }, $  where
\begin{equation}  \label{totalweight}
w_{ Y } = \sum _ { i \in Y } w_{ i} + u .
\end{equation}
For the quotient of ${\mathrm{a}}\Sigma ^ { \times } $ by ${\mathbb{C}} ^ {
\times } $ to exist, we have to assume that $w _ { Y } > 0 $ for all $Y \in
\mathcal{\ Y } $, which, as remarked earlier, can be done by taking $u $ to
be sufficiently large. The ambient weighted projective space of ${\mathrm{w}}%
\Sigma $ is thus
\begin{equation*}
{\mathbb{P}} ( w _ { Y } ~ | ~ {\ Y \in \mathcal{\ Y } } ) .
\end{equation*}

\subsection{Open Charts and Cell Decomposition of ${\mathrm{w}}\Sigma$}

\label{sec:CellDec}

We briefly recall the Bruhat order on the Weyl group $S_{n}$. Let $\sigma
\in S_{n } $. We define the \emph{inversions} of $\sigma $ to be
\begin{equation*}
\text{ Inv} ( \sigma ) : = \left \{ (i, j ) | i < j , \sigma ( i ) > \sigma
( j ) \right \} .
\end{equation*}
The \emph{length} of $\sigma$ is defined by $l(\sigma):=|\text{Inv}(\sigma)|$%
. Then, the permutation $\sigma_{0 } = ( n , n-1, \ldots, 1 ) $(written in
the one-line notation) has the highest length $\binom{n}{2} $. We say that $%
\sigma $ is \emph{covered} by $\tau \in S _ { n } $ and, denote by $\sigma
\to \tau $, if there is a transposition $(ij) $ such that $(ij)\cdot\sigma =
\tau $ and $l ( \tau ) = l ( \sigma ) + 1 $. Note that $(ij)$ acts on $\sigma
$ by switching indices $i$ and $j$ wherever they appear in $\sigma$.
We say that $\sigma \prec \tau $, if there is a sequence $\sigma \to \sigma
^ { ( 0 ) } \to \sigma ^ { (1) } \to \dotsb \to \sigma ^ { (s) } \to \tau $.
The reflexive closure $\preceq $ of the relation induced by $\prec $ is a
partial order on $S_{n} $ known as the \emph{Bruhat order}. The permutation $%
\sigma _{0} $ is then the unique maximal element under this order, and the
identity permutation $\mathrm{id} $ is the unique minimum.

Let $W^{\prime }:=W/W_P$ be the quotient of the Weyl group $W$. The quotient
$W^{\prime }$ corresponding to the highest weight \eqref{chi} is given by
\begin{equation*}
W^{\prime }=S_n/(S_{d_1}\times S_{d_2-d_1}\times \cdots\times
S_{d_r-d_{r-1}}).
\end{equation*}
Each coset $[\sigma] \in W^{\prime }$ has a unique representative of minimal
length and the length of this representative is same in $W$ and $W^{\prime }$%
, i.e. $l([\sigma])=l(\sigma)$. We denote by ${\mathrm{Inv}}_P(\sigma)$ all
inversions of $\sigma$ modulo $W_P.$ Thus, the Bruhat order on $W$ induces
the Bruhat order on $W^{\prime }$. The dimension of ${\mathrm{w}}\Sigma$ is
equal to the length of longest element $\sigma_0^{\prime }$ in $W^{\prime }$%
. The set of minimal length representatives of all cosets in $W^{\prime }$
will be denoted by $W^P.$

Recall the Bruhat decomposition of $G$ is given by
\begin{equation*}
G=\mathop{\coprod}\limits_{\sigma \in W}B\sigma B.
\end{equation*}
This induces the cell decomposition of $\Sigma$ into Schubert cells $%
\Omega_\sigma$ parameterized by $W^P$, i.e.,
\begin{equation*}
\Sigma=\mathop{\coprod}\limits_{\sigma \in W^P}B\sigma P/P.
\end{equation*}
Let $e_{\sigma}$ denote the flag in $\Sigma$ corresponding to $\sigma\in W^P$
in the given basis $e_1,\ldots, e_n$. Explicitly,
\begin{equation*}
e_{\sigma}=\mathop{\otimes}\limits_{i=1}^r \mathop{\wedge}%
\limits_{j}e_{\sigma(j)},
\end{equation*}
for $j\in I_i=\{\sigma(1),\ldots,\sigma(d_i)\}$. Each $\sigma \in W^P $
corresponds to a unique tableau $Y \in \mathcal{Y} $ as follows: fill the $i$%
th column of $Y$ with entries from $\{\sigma(1),\ldots,\sigma(d_r-i+1)\}$
such that they are in increasing order. For instance, $\sigma=\mathrm{id}$
corresponds to $Y_0$.

\begin{ej}
Consider the complete flag variety ${\mathcal{F}l}(4)=G/B$. Then $%
\sigma=(1,4,2,3)\in S_4$ corresponds to
\begin{equation*}
e_{\sigma}=e_1\otimes e_1\wedge e_4\otimes e_1\wedge e_4 \wedge e_2
\end{equation*}
and the corresponding tableau $Y$ along with its weight is given as follows:%
\newline

\begin{equation*}
\young(111,24,4)
\end{equation*}

\begin{equation*}
t_{1}^{3} t_{2}t_4^2 .
\end{equation*}
\end{ej}

Consider the (generalized) $G $-equivariant Pl\"ucker embedding of flag
variety
\begin{equation*}
\Sigma=G/P\hookrightarrow {\mathbb{P}} ( V _ { \chi } ), \qquad \qquad g
\cdot e_ { \mathrm{id} } \mapsto g \cdot e_{Y_{0} } .
\end{equation*}
For each $\sigma\in W^P$, $e_{\sigma}$ corresponds to $e_{Y}$, thus we
identify them. Let
\begin{equation*}
U_\sigma=\Sigma\cap U _ { Y } ,
\end{equation*}
where $U_{ Y } $ is the open set of all points in ${\mathbb{P}}( V _ { \chi
} ) $ with a non-zero $e _ { Y } $- coordinate. Then, by \cite%
{C&RweightedGrass} each $U _ { \sigma } $ is $T $-equivariantly isomorphic
to ${\mathbb{C}}^{ l(\sigma_0^{\prime }) }$, where $\sigma_0^{\prime }$ is
the maximal element in $W^P$. The union over all $\sigma\in W^P$ covers $%
\Sigma $. Moreover, the $T$-fixed points of $\Sigma$ are also parameterized
by $W^P$ and their number is equal to $|W^P|$ .

We now describe the open charts for ${\mathrm{w}} \Sigma $. Let
\begin{equation*}
K : = T \times {\mathbb{C}} ^ { \times }
\end{equation*}
be the $n + 1 $-torus inside $G \times {\mathbb{C}} ^ { \times } $. Let
\begin{equation*}
\Pi _ { {\mathrm{w}} } : {\mathrm{a}}\Sigma ^ { \times } \to {\mathrm{w}}%
\Sigma, \, \, \, \pi_{ {\mathrm{w}} } : K \to T _ { {\mathrm{w}} }
\end{equation*}
be the quotient maps. For the straight flag manifold $\Sigma$, we denote the
corresponding maps by $\Pi:{\mathrm{a}}\Sigma\rightarrow \Sigma $ and $%
\pi:K\rightarrow T $ respectively. By definition, $\Pi_{{\mathrm{w}}} $
makes ${\mathrm{a}}\Sigma ^ { \times }$ a $\pi_{{\mathrm{w}}} $-equivariant $%
{\mathbb{C}}^{\times} $-principal bundle over ${\mathrm{w}}\Sigma$. Indeed,
a $\pi $-equivariant trivialization is given by
\begin{equation*}
\varphi_{ \sigma } : {\mathrm{a}} U_{\sigma} \to U_{ \sigma} \times {\mathbb{%
C}} ^ { \times } \, \, , x \mapsto ( \Pi ( x ) , x _ { \sigma } ) ,
\end{equation*}
where ${\mathrm{a}} U_{\sigma}:=\Pi^{-1}(U_{\sigma})$ and the $K $-action on
$U _ { \sigma } \times {\mathbb{C}}^ { \times } $ is defined by factoring
through $T $ via $\Pi $ on the first component, and by $\pi_{\sigma } $(the
projection on the $\sigma$-coordinate) on the second. Let
\begin{equation*}
{\mathrm{w}} U _ { \sigma } : = {\mathrm{a}} U _ { \sigma } / {\mathbb{C}}
^{ \times }
\end{equation*}
and consider the map induced by $\varphi_{\sigma}$ between quotients
\begin{equation*}
\varphi_{ \sigma} ^ { {\mathrm{w}} } : {\mathrm{w}} U _ { \sigma} \to ( U _
{ \sigma } \times {\mathbb{C}} ^ { \times } ) / {\mathbb{C}} ^ { \times } .
\end{equation*}
The stabilizer of the ${\mathbb{C}} ^ {\times } $ on the second component of
$U _ { \sigma } \times {\mathbb{C}} ^ { \times } $ is the cyclic group $\mu
_ { \sigma } : = \left \{ t \in {\mathbb{C}} ^ { \times } | t ^ { w _ {
\sigma } } = 1 \right \},$ where $w_{\sigma}=w_Y$ and thus, we obtain the
following isomorphisms
\begin{equation*}
{\mathrm{w}} U _ { \sigma } \cong U _ { \sigma } / \mu _ { \sigma } \cong {%
\mathbb{C}} ^ { l(\sigma_0 ^{\prime }) } / \mu _ { \sigma },
\end{equation*}
that are all $T _ { {\mathrm{w}} } $-equivariant.

In a similar fashion, one can lift the Schubert cells $\Omega_{\sigma }$ via
$\Pi $ to ${\mathrm{a}} \Omega_{\sigma } \cong \Omega _ { \sigma } \times {%
\mathbb{C}} ^ { \times } \subset {\mathrm{a}}\Sigma^\times $, which has same
dimension as ${\Omega}_\sigma$ and its irreducibility follows from that of ${%
\Omega}_\sigma$. Then the $K$-invariant cell decomposition

\begin{equation*}
{\mathrm{a}}\Sigma^\times=\bigsqcup_{\sigma \in W^P} {\mathrm{a}}{\Omega}%
_\sigma
\end{equation*}
is obtained from the Bruhat decomposition of $\Sigma=\bigsqcup_{\sigma \in
W^P}\Omega_{\sigma}.$ Similarly, the weighted flag variety ${\mathrm{w}}%
\Sigma$ has a $T_{{\mathrm{w}}}$-invariant weighted cell decomposition

\begin{equation*}
{\mathrm{w}}\Sigma=\bigsqcup_{\sigma \in W^P}{\mathrm{w}}{\Omega}_\sigma
\text{ \;such that \;} {\mathrm{w}}{\Omega}_\sigma:={\mathrm{a}}{\Omega}%
_\sigma/{\mathbb{C}}^{\times}.
\end{equation*}
Using the chart $\varphi_{\sigma}^{{\mathrm{w}}}$, we have
\begin{equation*}
{\mathrm{w}}{\Omega}_{\sigma}\cong {\Omega}_{\sigma}/\mu_{\sigma}.
\end{equation*}

We will briefly recall the main idea of the Borel-Moore homology from \cite%
{fulton1997young} and \cite{borel1960homology}.

\begin{defi}
If $X$ is embedded in a closed subspace of Euclidean space $\mathbb{R}^n$,
then its rational Borel-Moore homology is given by
\begin{equation*}
H_i^{BM}(X):=H^{n-i}(\mathbb{R}^n, \mathbb{R}^n-X).
\end{equation*}
In particular, for any oriented $n$-dimensional manifold $M$, we have
\begin{equation*}
H_i^{BM}(M)=H^{n-i}(M,M-M)=H^{n-i}(M),
\end{equation*}
by \cite[Appendix B.2-(26)]{fulton1997young}.
\end{defi}

This formula can be extended to the case of any oriented, rationally smooth
variety (see \cite[11.4.3]{cox2011toric}). Given the cell decomposition of ${%
\mathrm{w}}\Sigma$, we obtain the following isomorphism \cite[Ex. 6,
Appendix B.2]{fulton1997young}):
\begin{equation}  \label{BM}
H_i^{BM}({\mathrm{w}}\Sigma)\cong \bigoplus_{\sigma\in W^P}H_i^{BM}({\mathrm{%
w}}{\Omega}_{\sigma}).
\end{equation}
Now,
\begin{equation}  \label{BM2}
\begin{array}{rcl}
H_i^{BM}({\mathrm{w}}{\Omega}_{\sigma}) & = & H^{2(\dim{\mathrm{w}}{\Omega}%
_{\sigma})-i}({\mathrm{w}}{\Omega}_{\sigma}) \\
& = & H^{2(\dim {\mathrm{w}}{\Omega}_{\sigma})-i}({\Omega}%
_{\sigma}/\mu_{\sigma}) \\
& \cong & H^{2(\dim {\mathrm{w}}{\Omega}_{\sigma})-i}({\Omega}%
_{\sigma})^{\mu_{\sigma}} \\
& \cong & H^{2(\dim {\mathrm{w}}{\Omega}_{\sigma})-i}({\Omega}_{\sigma})%
\end{array}%
,
\end{equation}
where the first equality follows from the statements given above, since each
cell is locally the quotient of a Euclidean space modulo a finite group,
hence rationally smooth \cite[11.4.4]{cox2011toric}. The first isomorphism
follows from the classical result (see \cite[Theorem 2.4, Chapter III]%
{bredon1972introduction}). The second isomorphism comes from the fact that
as $\mu_{\sigma}$ acts through the connected group $\mathbb{C}^{\times}$,
hence its induced action on cohomology must be trivial. By Equations (4) and
(5), the odd rational cohomology of ${\mathrm{w}}\Sigma$ vanishes.

\begin{prop}
\label{Cohomology}
\begin{equation*}
H^i({\mathrm{w}}\Sigma)=\left\{%
\begin{array}{cc}
{\mathbb{Q}}^{h} & i=2k \\
0 & \mathrm{otherwise}%
\end{array}%
\right.
\end{equation*}
where $h$ is the number of elements in $W^P$ of length $k$.
\end{prop}

\begin{proof}
The proof follows from \eqref{BM} and \eqref{BM2}.
\end{proof}

Thus we have the following theorem:

\begin{teo}
There is an $H^*(BT_{\mathrm{w}})$-module isomorphism
\begin{equation*}
H^ { * } _ {T _ { {\mathrm{w}} } } ( {\mathrm{w}}\Sigma ) \cong H ^ { * } (
B T _ { {\mathrm{w}} } ) \otimes _ { {\mathbb{Q}} } H ^ { * } ( {\mathrm{w}}%
\Sigma ),
\end{equation*}
where $H^*_{T_{{\mathrm{w}}}}(.)$ is the $T_{{\mathrm{w}}}$-equivariant
cohomology and $BT_{{\mathrm{w}}}$ is the classifying space of $T_{{\mathrm{w%
}}}$ (see \cite{borel2016seminar}). In fact, $H^ { * } _ {T _ { {\mathrm{w}}
} } ( {\mathrm{w}}\Sigma )$ is a free module.
\end{teo}

\begin{proof}
The torus $T_{{\mathrm{w}}}$ is connected and the odd cohomology of ${%
\mathrm{w}}\Sigma$ vanishes. The result follows from \cite[Theorem 14.1]%
{goresky1997equivariant}.
\end{proof}


\begin{ej}
Consider the flag variety
\begin{equation*}
\Sigma=GL(4,{\mathbb{C}})/P= \{0\subset V_{d_1}\subset V_{d_2}\subset{%
\mathbb{C}}^4\}
\end{equation*}
where $d_1=1$ and $d_2=3$. Here the Weyl group $S_4$ has the corresponding
Parabolic subgroup $W_P=\langle s_2\rangle .$ The set of minimal length
representatives of cosets in $W^P$ is
\begin{equation*}
W^P=\{\mathrm{id},s_1,s_3,
s_1s_3,s_2s_1,s_2s_3,s_2s_1s_3,s_1s_2s_3,s_3s_2s_1,s_3s_2s_3s_1,s_1s_2s_3s_1,s_3s_1s_2s_3s_1\}.
\end{equation*}
The variety $\Sigma$ is five dimensional, thus by using Proposition \ref%
{Cohomology}, we have
\begin{equation*}
H^0({\mathrm{w}}\Sigma)=H^{10}({\mathrm{w}}\Sigma)={\mathbb{Q}},\; H^2({%
\mathrm{w}}\Sigma)=H^8({\mathrm{w}}\Sigma)={\mathbb{Q}}^2 \text{ and } H^4({%
\mathrm{w}}\Sigma)=H^6({\mathrm{w}}\Sigma)={\mathbb{Q}}^3.
\end{equation*}

\end{ej}

\subsection{Equivariant Cohomology Ring of ${\mathrm{w}}\Sigma $}

Since ${\mathrm{a}}\Sigma ^ { \times } $ is a smooth quasi-projective
variety inside ${\mathbb{C}} ^ { | \mathcal{\ Y } | } $, it has a natural
structure of a symplectic manifold given by $\omega = \sum _ { Y \in
\mathcal{\ Y } } d x _ { Y } \wedge d \overline{ x } _ { Y } $. Following
\cite[Section 9]{Kirwan}, we describe ${\mathrm{w}}\Sigma $ as a quotient of
a compact real symplectic submanifold of ${\mathrm{a}}\Sigma ^ { \times } $
by the hamiltonian action of compact real torus inside ${\mathbb{C}}^{
\times } $.

Let $S ^ { 1 } = \left \{ z \in {\mathbb{C}} ^ { \times } | | z | = 1 \right
\} $, $S _ {T } : = ( S ^ { 1 } ) ^ {n } $, $S _ { K } : = ( S ^ { 1} ) ^ {
n+1} $ be the real tori inside ${\mathbb{C}} ^ { \times } $, $T $ and $K $
respectively. We define the real torus inside $T_{{\mathrm{w}}}$ to be $%
S_{T_{{\mathrm{w}}}}:= S_{K}/ S^1.$

\begin{lema}
There is compact symplectic submanifold $M $ of ${\mathrm{a}}\Sigma ^ {
\times } $ such that ${\mathrm{a}}\Sigma ^ { \times }$ $K $-equivariantly
deformation retracts to $M $ and
\begin{equation*}
M / S ^ { 1 } \cong {\mathrm{w}}\Sigma .
\end{equation*}
\end{lema}

\begin{proof}
The $S^ {1} $ action on $V_ {\chi } \cong \mathbb{R} ^ { 2 | \mathcal{Y} | }
$ factors through the ${\mathbb{C}} ^ { \times } $ action defined earlier,
and since ${\mathrm{a}}\Sigma ^ { \times } $ is $S ^{1} $-invariant, the
action restricts to ${\mathrm{a}}\Sigma ^ { \times } $. The associated
vector field for this action is
\begin{equation*}
X = \sum _ { Y \in \mathcal{\ Y } } w _ { Y} ( - x ^ { 2 } _ { Y } \partial
x ^ {1} _ { Y } + x ^ { 1} _ { Y } \partial x ^ { 2 } _ { Y } )
\end{equation*}
where $x ^ { 1 } _ { Y } = \Re ( x_ { Y } ) $, $x ^ { 2 } _ { Y } = \Im ( x
_ { Y } ) $ and $w_Y$ is as defined in \eqref{totalweight}.

We identify $\mathbb{R} $ with both the Lie algebra of $S^{1} $ and its
dual. The duality pairing then becomes the standard inner product on $%
\mathbb{R} $. Let $\mu : V _ { \chi } \to \mathbb{C} $ be defined by
\begin{equation*}
\mu ( x ) = i \sum _{ Y \in \mathcal{Y} } w _ { Y } | x _ { Y } | ^ { 2 } .
\end{equation*}

Below we check that $\mu$ is a moment map with respect to which $X$ is
Hamiltonian.
\begin{align*}
\iota_X (\omega) & = 2i \sum _ { Y \in \mathcal{\ Y } } w_ { Y } ( x ^ { 2 }
_ { Y} d x ^ { 2} _ {Y } + x ^ { 1 } _{ Y } d x ^ { 1 } _{ Y } ) \\
& = 2 i \sum _ { Y \in \mathcal{Y} } \frac{1}{2} w_{Y} d ( ( x _{Y} ^{2} ) ^
{ 2} + ( x _{ Y } ^ {1} ) ^ {2} ) \\
& =i \sum _ { Y \in \mathcal{Y} }d(w_Y |x_Y|^2) \\
& = d \mu .
\end{align*}

Since ${\mathrm{a}}\Sigma^{\times} $ is an $S^{1} $-invariant symplectic
submanifold in $V_{\chi } $, the restriction of $\mu $ on ${\mathrm{a}}%
\Sigma ^ { \times } $ gives us a moment map with respect to which the
induced action on ${\mathrm{a}}\Sigma^{\times} $ is Hamiltonian.

Let $M = \mu ^ { - 1 } ( \lambda ) $ be a level surface for any regular
value $\lambda $ of $\mu $. Then, $M $ is a smooth $S _ { K } $-invariant
submanifold inside ${\mathrm{a}}\Sigma^ { \times } $. The map
\begin{equation*}
f : {\mathrm{a}}\Sigma ^ { \times } \times [0,1 ] \to {\mathrm{a}}\Sigma ^ {
\times }
\end{equation*}
given by
\begin{equation*}
((x_{Y})_{ Y \in \mathcal{\ Y } }, t ) \mapsto \big  (    ( 1 + t ( \sqrt{
\lambda / \mu ( x ) } - 1 \big      )                      ) x _ { Y } ) _ {
Y \in \mathcal{\ Y } }
\end{equation*}
is easily checked to be a deformation retraction from ${\mathrm{a}}\Sigma^ {
\times } $ to $M $. Furthermore, this deformation retraction is equivariant
with respect to the $S _ { K } $ action that is, each $f_{t} $ is a $S_{K } $%
-equivariant map for all $t\in[0,1]$.

Since ${\mathrm{a}}\Sigma ^ { \times } $ is $S_{K} $-invariant, the quotient
${\mathrm{w}}\Sigma ^ { \times} $ is $S_{ T_{{\mathrm{w}}} } $-invariant. As
the inclusion $\iota : M \to {\mathrm{a}}\Sigma ^ { \times } $ is
equivariant with respect to the inclusion $(S^1 )^{n+1} \hookrightarrow K $,
the induced quotient map $\overline { \iota } : M / S ^{1} \to {\mathrm{w}}%
\Sigma $ will be equivariant with respect to the inclusion $S _ { T _ { {%
\mathrm{w}} } } \hookrightarrow T _ { {\mathrm{w}} } $. Furthermore, it is a
homeomorphism as its inverse is given by 
\begin{equation*}
(x_{Y})_{Y \in \mathcal{Y} } \to ( x_{ Y } \sqrt { \lambda / \mu ( x ) } ) _
{ Y \in \mathcal{\ Y } } ,
\end{equation*}
which can easily be seen to be continuous.
\end{proof}

Using the above Lemma we obtain the following isomorphism which gives us the
equivariant cohomology ring of ${\mathrm{w}}\Sigma$.

\begin{teo}
The map $\overline{\iota}^*:H_{T_{{\mathrm{w}}}}^*({\mathrm{w}}\Sigma)\to
H^*_{S_{T_{{\mathrm{w}}}}}(M/S^1) $ is a ring isomorphism.
\end{teo}


\section{Schubert classes and GKM descriptions}

\label{S-GKM}

\subsection{Weighted Schubert classes}

Let us denote by $\tilde{\xi}_{\sigma}$ the equivariant Schubert classes
forming distinguished $H^*_T(BT)$-module basis for $H_T^*{(\Sigma)}$, see
\cite{borel2016seminar}. The classes $\tilde{\xi} _ { \sigma } $ are
obtained as the \emph{equivariant fundamental classes} of the $T $-invariant
Schubert varieties $X_{\sigma} $ which are closures of the corresponding
Schubert cells $\Omega_{ \sigma } $. We start with the following diagram.

\begin{center}
\begin{tikzcd}         [column sep = large, row sep = large]
   ET  \times _ { T }  \Sigma   \arrow[d] &   E K \times _ { K }  \a\Sigma^{\times} \arrow[l,   swap ,   "\Pi"] \arrow[r, "\Pi_{\w}"] \arrow[d] &   E T _ { \w }  \times _ { T _{ \w } }  \w\Sigma      \arrow[d] \\
BT  &  BK   \arrow    [r, "\pi_{\w}"] \arrow[l , swap   ,     "\pi"]  &  BT_{\w}
\end{tikzcd}
\end{center}

Applying the cohomology functor, we get the following diagram:

\begin{center}
\begin{tikzcd}  [column sep = large, row sep = large]
 H^{*}_{T} (\Si)   \arrow[r,"\Pi^*"]  &   H^{* } _ {K} (     \a\Si^{\times}  )  &     H ^ { * }   _ {T _{ \w} }  (   \w\Si    )           \arrow[l,swap, "\Pi_{\w}^*"] \\

    H^*_T(BT)          \arrow[r,   "\pi^{*}"] \arrow[u]  &            H^*_K(BK)    \arrow[u]   &   H^*_{T_{\w}}(BT_{\w})    \arrow[l     , swap   ,     "\pi_{\w} ^ { *   }    "]                                       \arrow[u]
\end{tikzcd}
\end{center}

We take $\tilde{\xi}_{\sigma}= [\Omega_{\sigma}]_T\in
H_T^{2l(\sigma)}(\Sigma)$, then the Schubert cycles for the punctured affine
cone ${\mathrm{a}}\Sigma^\times$ and the weighted flag orbifold ${\mathrm{w}}%
\Sigma$, can be described via maps $\Pi^*$ and $\Pi_{{\mathrm{w}}}^*$
respectively as follows, by the same argument as in \cite{A&MweightedGrass}
for the case of Grassmannians:
\begin{equation*}
\begin{array}{cc}
{\mathrm{a}}\tilde{\xi}_{\sigma}:=\Pi^*(\tilde{\xi}_{\sigma})\in
H_K^{2l(\sigma)}({\mathrm{a}}\Sigma ^{\times}) & {\mathrm{w}}\tilde{\xi}%
_{\sigma}:=(\Pi_{{\mathrm{w}}}^*)^{-1}({\mathrm{a}}\tilde{\xi}_{\sigma})\in
H_{T_{{\mathrm{w}}}}^{2l(\sigma)}({\mathrm{w}}\Sigma)%
\end{array}%
.
\end{equation*}

\subsection{ GKM description of $H^*_{T_{\mathrm{w}}}({\mathrm{w}}\Sigma)$}

Let the $T$-equivariant cohomology $H ^*( BT ) $ of a point be denoted by ${%
\mathbb{Q}}[T^*]\cong \mathbb{Q}[y_1,\ldots,y_n]$. Let
\begin{equation*}
{\mathbb{Q}}[K^*]:=H^*(BK) ={\mathbb{Q}}[y_1,\ldots,y_n,z].
\end{equation*}
Since $K=T\times \mathbb{C}^{\times}$, one can take the basis of ${\mathbb{Z}%
}$-linear functionals $\{y_1,\ldots,y_n,z\}$ on the integral sublattice of $%
\mathrm{Lie}(K)^*$.

Define $\rho^*: \mathrm{Lie}(K)^*\rightarrow \mathrm{Lie}({\mathbb{C}}%
^{\times})^*={\mathbb{Q}}[\delta]$ by
\begin{equation*}
y_i\mapsto w_i\delta \mathrm{\ and\ } z\mapsto -u\delta
\end{equation*}
By definition of $T_{{\mathrm{w}}}$ we have $\ker(\rho^*)=\mathrm{Lie}(T_{{%
\mathrm{w}}})^*$. Let $y_1^{{\mathrm{w}}},\ldots, y_n^{{\mathrm{w}}}$ denote
the $\mathbb{Z}$-basis of $\mathrm{Lie}(T_{{\mathrm{w}}})^*.$ Since $T_{{%
\mathrm{w}}}$ is defined as the quotient of $K$, thus we define $y_i^{%
\mathrm{w}}$ as elements of $\ker(\rho^*)$ by
\begin{equation}  \label{yw}
y_i^{{\mathrm{w}}}:= y_i + \frac{w_i}{u}z, \mbox{ for all } i=1,\ldots,n.
\end{equation}

Let
\begin{equation*}
{\mathbb{Q}}[T_{{\mathrm{w}}}^*]:=H^*(BT_{{\mathrm{w}}})={\mathbb{Q}}[y_1^{{%
\mathrm{w}}},\ldots, y_n^{{\mathrm{w}}}]\subset {\mathbb{Q}}[K^*].
\end{equation*}

For each $\sigma\in W^P,$ we use the following notation in the rest of the
paper,
\begin{equation}\label{4}
\arraycolsep=1pt\def\arraystretch{2}
\begin{array}{l}

y_{\si}:=r \left(y_{\si_{1}}+\cdots + y_{\si_{d_1}}\right) +(r-1)\left(y_{\si_{d_1+1}}+\cdots+y_{\si_{d_2}}\right)+\cdots+\left(y_{\si_{d_{r-1}+1}}+\cdots+ y_{\si_{d_r}}\right),\\

y^{\w}_{\si}:=r \left(y^{\w}_{\si_{1}}+\cdots + y^{\w}_{\si_{d_1}}\right) +(r-1)\left(y^{\w}_{\si_{d_1+1}}+\cdots+y^{\w}_{\si_{d_2}}\right)+\cdots+\left(y^{\w}_{\si_{d_{r-1}+1}}+\cdots+ y^{\w}_{\si_{d_r}}\right),\\

w_{\si}:=r \left(w_{\si_{1}}+\cdots + w_{\si_{d_1}}\right) +(r-1)\left(w_{\si_{d_1+1}}+\cdots+w_{\si_{d_2}}\right)+\cdots+\left(w_{\si_{d_{r-1}+1}}+\cdots+ w_{\si_{d_r}}\right)+u,\\
\end{array}
\end{equation}

where $\sigma_i=\sigma(i)$. The fixed points under the action of $T$ on $%
\Sigma$ are $[e_{\sigma}]$, for $\sigma\in W^P$. We also denote the $T_{{%
\mathrm{w}}}$-fixed points of ${\mathrm{w}}\Sigma$ by the same notation $%
[e_{\sigma}]$. By equivariant formality of $\Sigma$ (see for instance in
\cite{guillemin2006gkm}), the restriction map to the fixed points is
injective,
\begin{equation}  \label{(1)}
H_T^*(\Sigma)\to \mathop{\bigoplus}\limits_ {\sigma\in W^P} {\mathbb{Q}}%
[T^*]; \,\,\,\,\, \alpha\mapsto (\alpha _{\sigma})_{\sigma\in W^P}.
\end{equation}
By \cite{goresky1997equivariant}, the image of the above map is given by
\begin{equation*}
\Big\{ \alpha=(\alpha_{\sigma})_{\sigma\in W^P}\in \mathop{\bigoplus}%
\limits_ {\sigma\in W^P} {\mathbb{Q}}[T^*]\ \ \Big| \ \
\alpha_{\sigma}-\alpha_{\tau}\in \langle y_i-y_j\rangle= \langle
y_{\sigma}-y_{\tau}\rangle;\mbox{ whenever } \sigma=(ij)\tau \Big\}.
\end{equation*}

Since the fixed points of the $T_{{\mathrm{w}}}$-action are the images of $%
[e_{\sigma}]$ in ${\mathrm{w}}\Sigma$ and given the isomorphism $H^*_{T_{{%
\mathrm{w}}}}([e_{\sigma}])\cong {\mathbb{Q}}[T_{{\mathrm{w}}}^*]$, we have
the following restriction map
\begin{equation}  \label{(2)}
H^*_{T_{{\mathrm{w}}}}({\mathrm{w}}\Sigma)\to \mathop{\bigoplus}\limits_
{\sigma\in W^P} {\mathbb{Q}}[T_{{\mathrm{w}}}^*]; \,\,\,\,\, \gamma\mapsto
(\gamma _{\sigma})_{\sigma\in W^P}.
\end{equation}

For any $\sigma\in W^P$, the image of $[e_{\sigma}]$ in ${\mathrm{a}}%
\Sigma^{\times}$ is a $\mathbb{C}^{\times}$-orbit of $e_\sigma$. Denote by $%
K_{\sigma}$, kernel of the map $K\to \mathbb{C}^{\times},$ given by
\begin{equation*}
(t_1,\ldots,t_n,s)\mapsto s^{-1}t^{\sigma}; \,
t^{\sigma}=(t_{\sigma_1}\ldots t_{\sigma_{d_1}})^r(t_{\sigma_{d_1+1}}\ldots
t_{\sigma_{d_2}})^{r-1}\ldots (t_{\sigma_{d_{r-1}+1}}\ldots
t_{\sigma_{d_r}}).
\end{equation*}
By definition, $K_{\sigma}$ is the set of those elements which fix $%
e_{\sigma}\in {\mathrm{w}}\Sigma$. Thus the $K_\sigma$-equivariant
cohomology of the fixed point $[e_\sigma]$ is
\begin{equation*}
H^*_{K_{\sigma}}(e_{\sigma})\cong {\mathbb{Q}}[K_{\sigma}^*],
\end{equation*}
where  ${\mathbb{Q}}[K_{\sigma}^*]= {\mathbb{Q}}[K]/ \langle y_{\sigma}-z
\rangle $. Indeed, this is true because $K_{\sigma}$ is the isotropy
subgroup of $\sigma$ in $K$ and due to the following isomorphism
\begin{equation*}
\mathrm{Lie}(K_{\sigma})^*\cong \mathrm{Lie}(K)^*/\langle y_{\sigma}-z
\rangle .
\end{equation*}
The restriction map for ${\mathrm{a}}\Sigma^\times$ thus reads
\begin{equation}  \label{(3)}
H^*_K({\mathrm{a}}\Sigma^\times)\to \mathop{\bigoplus}\limits_ {\sigma\in
W^P} {\mathbb{Q}}[K_{\sigma}^\star]\cong \mathop{\bigoplus}\limits_
{\sigma\in W^P} {\mathbb{Q}}[K^*]/\langle y_{\sigma}-z \rangle;
\,\,\,\,\,Q\mapsto (Q_{\sigma})_{\sigma\in W^P}
\end{equation}
The map (\ref{(3)}) is injective due to injectivity of (\ref{(1)}) and (\ref%
{(2)}) along with the commutativity of the following diagram

\begin{equation}  \label{CommuDiag}
\xymatrix{ H_{T}^*(\Si) \ar[rr]^{\Pi^{*}}\ar[d] &&
H^*_K(\a\Si^{\times})\ar[d] && \ar[ll]_{\Pi_\w^*} H^*_{T_\w}(\w\Si)\ar[d]\\
\mathop{\bigoplus}\limits_ {\si\in W^P} \QQ[T^*] \ar[rr]^{\kappa_{\si}^{*}}
&& \mathop{\bigoplus}\limits_ {\si\in W^P} \QQ[K_{\si}^*] &&
\ar[ll]_{\kappa_{\si} ^ {\w * } } \mathop{\bigoplus}\limits_ {\si\in W^P}
\QQ[T_{\w}^*]}
\end{equation}
where the maps giving the bottom arrows are induced by $\kappa_{\sigma}:
K_{\sigma}\hookrightarrow K\to T$ and $\kappa^{{\mathrm{w}}}_{\sigma}:
K_{\sigma}\hookrightarrow K\to T_{{\mathrm{w}}} $ defined by sending each
element to its class modulo $\langle y_{\sigma}-z \rangle$.

\begin{teo}
\label{thm4.1} The equivariant cohomology of ${\mathrm{w}}\Sigma$ has the
following GKM description
\begin{equation*}
\Big\{ \gamma=(\gamma_{\sigma})_{\sigma\in W^P}\in \mathop{\bigoplus}%
\limits_ {\sigma\in W^P} {\mathbb{Q}}[T_{{\mathrm{w}}}^*] \Big| %
\gamma_{\sigma}-\gamma_{\tau}\in \langle w_{\tau}y_{\sigma}^{{\mathrm{w}}%
}-w_{\sigma}y_{\tau}^{{\mathrm{w}}} \rangle;\mathrm{\ whenever }\;
\sigma=(ij)\tau \Big\}.
\end{equation*}
\end{teo}

\begin{proof}
In order to show the equivalence of GKM conditions we need to check if the
following isomorphism is valid:
\begin{equation*}
{\mathbb{Q}}[T_{{\mathrm{w}}}^*]/\langle w_{\tau}y_{\sigma}^{{\mathrm{w}}%
}-w_{\sigma}y_{\tau}^{{\mathrm{w}}} \rangle\cong {\mathbb{Q}}[K^*]/\langle
y_{\sigma}-z,y_{\tau}-z \rangle.
\end{equation*}
The first of the following isomorphisms is implied by the isomorphism $%
\kappa^{{\mathrm{w}}^*}_{\sigma}$,
\begin{equation*}
{\mathbb{Q}}[T_{{\mathrm{w}}}^*]/\langle w_{\tau}y_{\sigma}^{{\mathrm{w}}%
}-w_{\sigma}y_{\tau}^{{\mathrm{w}}} \rangle\cong {\mathbb{Q}}[K^*]/\langle
y_{\sigma}-z,w_{\tau}y_{\sigma}^{{\mathrm{w}}}-w_{\sigma}y_{\tau}^{{\mathrm{w%
}}} \rangle\cong {\mathbb{Q}}[K^*]/\langle y_{\sigma}-z,y_{\tau}-z \rangle.
\end{equation*}
The second isomorphism follows because
\begin{equation*}
w_{\tau}y_{\sigma}^{{\mathrm{w}}}-w_{\sigma}y_{\tau}^{{\mathrm{w}}%
}=-w_{\sigma}(y_{\tau}-z) \mbox { in } {\mathbb{Q}}[K^*]/\langle
y_{\sigma}-z \rangle .
\end{equation*}
\end{proof}

\begin{teo}
\label{thm4.2} The equivariant cohomology of ${\mathrm{a}}\Sigma^{\times}$
has the following GKM description
\begin{equation*}
\Big\{ Q=(Q_{\sigma})_{\sigma\in W^P}\in \mathop{\bigoplus}\limits_
{\sigma\in W^P} {\mathbb{Q}}[K^*] \Big| Q_{\sigma}=Q_{\tau}\;\;\mathrm{\ in }%
\;\; {\mathbb{Q}}[K^*]/\langle y_{\sigma}-z,y_{\tau}-z \rangle;\;\mathrm{\
whenever }\;\;\sigma=(ij)\tau \Big\}.
\end{equation*}
\end{teo}

We have proven the equivalence of GKM conditions for $\kappa_{\sigma}^{{%
\mathrm{w}}*}$ above, the equivalence for $\kappa_{\sigma}^*$ follows as a
special case of Theorem \ref{thm4.1}.

Next we express the Schubert classes using the GKM descriptions stated
above. It is known for partial flag manifold $\Sigma$ \cite[Corollary 4.5]{Jul09}, that for
any $\sigma\in W^P,$

\begin{equation}  \label{GKMstraight}
\tilde{\xi}_{\sigma}|_{\tau}=\Bigg\{
\begin{array}{cc}
\mathop{\prod}\limits_{(i,j)\in {\mathrm{Inv}}_P(\sigma)}(y_j-y_i) &
\tau=\sigma \\
0 & \tau \nsucceq \sigma%
\end{array}%
\end{equation}
Using the definitions in (\ref{4}), we write the following equivalent
description
\begin{equation}  \label{GKMstraight2}
\tilde{\xi}_{\sigma}|_{\tau}=\Bigg\{
\begin{array}{cc}
h_{\sigma}\mathop{\prod}\limits_{(i,j)\in {\mathrm{Inv}}_P(\sigma)}(y_{%
\sigma}-y_{(ij)\sigma}) & \tau=\sigma \\
0 & \tau \nsucceq \sigma%
\end{array}%
,
\end{equation}
where $h_{\sigma}=\mathop{\prod}\limits_{(i,j)\in {\mathrm{Inv}}%
_P(\sigma)}(q-p)^{-1}$, where $p$ and $q$ are chosen so that $i\in
[d_p+1,d_{p+1}]$ and $j\in [d_q+1,d_{q+1}]$.\newline
This description along with the commutative diagram \eqref{CommuDiag} and
the formulae for Schubert classes in ${\mathrm{a}}\Sigma^{\times}$ and ${%
\mathrm{w}}\Sigma$ imply the following two propositions.

\begin{prop}
The Schubert classes in ${\mathrm{a}}\Sigma^{\times}$ have the following GKM
description
\begin{equation}  \label{GKMcone}
{\mathrm{a}}\tilde{\xi}_{\sigma}|_{\tau}=\Bigg\{
\begin{array}{cc}
h_{\sigma}\mathop{\prod}\limits_{(i,j)\in {\mathrm{Inv}}_P(%
\sigma)}(z-y_{(ij)\sigma}) & \tau=\sigma \\
0 & \tau \nsucceq \sigma%
\end{array}%
.
\end{equation}
\end{prop}

\begin{prop}
The Schubert classes in ${\mathrm{w}}\Sigma$ have the following GKM
description
\begin{equation}  \label{GKMweighted}
{\mathrm{w}}\tilde{\xi}_{\sigma}|_{\tau}=\Bigg\{
\begin{array}{cc}
h_{\sigma}\mathop{\prod}\limits_{(i,j)\in {\mathrm{Inv}}_P(\sigma)}(\frac{%
w_{(ij)\sigma}}{w_{\sigma}}y_{\sigma}^{{\mathrm{w}}}-y^{\mathrm{w}}%
_{(ij)\sigma}) & \tau=\sigma \\
0 & \tau \nsucceq \sigma%
\end{array}%
.
\end{equation}
\end{prop}

For simple transpositions $s_d\in W^P$ one has the simple formula,
\begin{equation}  \label{sd}
\tilde{\xi}_{s_d}|_{\tau}=\sum_{j=1}^{d}y_{\tau(j)}-\sum_{j=1}^{d} y_j.
\end{equation}
This can easily be checked for all $\tau\in W^P$. Using the upper
triangularity of weighted Schubert classes, as above, one obtains a basis
for the equivariant cohomology of the weighted flag orbifold ${\mathrm{w}}%
\Sigma$.

\begin{prop}
A $H^*(BT_{{\mathrm{w}}})$-module basis of $H^*_{T_{{\mathrm{w}}}}({\mathrm{w%
}}\Sigma)$ is given by the weighted Schubert classes $\{{\mathrm{w}}\tilde{%
\xi}_{\sigma}\}_{\sigma\in W^P}$.
\end{prop}

\section{ Chevalley's formula for ${\mathrm{w}} \Sigma$}

\label{S-Pieri}

Chevalley's formula computes the product of any Schubert cycle with a
divisor class corresponding to a simple root. We compute the weighted
version of the Chevalley's formula.  We also give the
polynomial description for the equivariant cohomology ring of ${\mathrm{w}}%
\Sigma$ and polynomial representative for a weighted Schubert class,
inspired by \cite{abe2015schur}.

\subsection{Weighted Chevalley's Formula}

Using \eqref{yw} and the Schubert class \eqref{sd} for simple transposition $s_d \in W^P$, we have
\begin{equation}  \label{asd}
{\mathrm{a}}\tilde{\xi}_{s_d}|_{\tau}=\sum_{j=1}^{d}\big(y^{{\mathrm{w}}%
}_{\tau(j)}- y^{{\mathrm{w}}}_j-\frac{w_{\tau(j)}-w_j}{u}z\big).
\end{equation}


 We  give an equivariant
Chevalley's formula for the weighted flag orbifold as follows.

\begin{teo}
\label{Th:Chev}
For any simple reflection $s_d\in W^P$ and $\sigma\in W^P$ the weighted
Chevalley's formula is given by
\begin{equation}  \label{Pieri}
{\mathrm{w}}\tilde{\xi}_{s_d}{\mathrm{w}}\tilde{\xi}_{\sigma}= \sum_{j=1}^{d}%
\big(y^{{\mathrm{w}}}_{\sigma(j)}- y^{{\mathrm{w}}}_j-\frac{w_{\sigma(j)}-w_j%
}{u}z\big){\mathrm{w}}\tilde{\xi}_{\sigma}+\mathop{\sum_{i\leq d< j}}%
\limits_{l((ij)\sigma)=l(\sigma)+1} {\mathrm{w}}\tilde{\xi}_{(ij)\sigma}.
\end{equation}
\end{teo}

\begin{proof}
We have the following formula of Kostant-Kumar \cite[Proposition 4.30]%
{kostant1986nil} in the case of partial flag variety $\Sigma$
\begin{equation*}
\tilde{\xi}_{s_d}\tilde{\xi}_{\sigma}= \tilde{\xi}_{s_d}|_{\sigma}\tilde{\xi}%
_{\sigma}+\mathop{\sum_{i\leq d_k< j}}\limits_{l((ij)\sigma)=l(\sigma)+1}%
\tilde{\xi}_{(ij)\sigma}.
\end{equation*}
By applying $\Pi^*$ to the above equation we get
\begin{equation*}
{\mathrm{a}}\tilde{\xi}_{s_d}{\mathrm{a}}\tilde{\xi}_{\sigma}={\mathrm{a}}%
\tilde{\xi}_{s_d}|_{\sigma}{\mathrm{a}}\tilde{\xi}_{\sigma}+%
\mathop{\sum_{i\leq d_k< j}}\limits_{l((ij)\sigma)=l(\sigma)+1}{\mathrm{a}}%
\tilde{\xi}_{(ij)\sigma}.
\end{equation*}
An application of $(\Pi_{{\mathrm{w}}}^*)^{-1}$ along with the use of %
\eqref{asd} we obtain the result.
\end{proof}

We define  a special Schubert class  $\tilde{\xi}_{{\mathrm{div}}}$ given by \eqref{div}, which corresponds to length one elements \(s_{d}\) in $W^P$. We use it to prove  Proposition  \ref{prop:div}.
\begin{equation}
\label{div}
\tilde{\xi}_{{\mathrm{div}}} :=\mathop{\sum}\limits_{i=1}^{r}\tilde{\xi}%
_{s_{d_i}}.
\end{equation}
Recall that for $\sigma\in W^P$,
\begin{equation*}
y_{\sigma}=r \left(y_{\sigma_{1}}+\cdots + y_{\sigma_{d_1}}\right)
+(r-1)\left(y_{\sigma_{d_1+1}}+\cdots+y_{\sigma_{d_2}}\right)+\cdots+%
\left(y_{\sigma_{d_{r-1}+1}}+\cdots+ y_{\sigma_{d_r}}\right)
\end{equation*}
By using the GKM description for simple transpositions and \eqref{asd}, we
get
\begin{equation*}
\begin{array}{rcl}
\tilde{\xi}_{{\mathrm{div}}}|_{\tau} & = & \tilde{\xi}_{s_{d_1}}|_{\tau}+%
\dotsb+\tilde{\xi}_{s_{d_r}}|_{\tau} \\
& = & \mathop{\sum}\limits_{j=1}^{d_1}y_{\tau(j)}-\mathop{\sum}%
\limits_{j=1}^{d_1} y_j+\dotsb+\mathop{\sum}\limits_{j=1}^{d_r}y_{\tau(j)}-%
\mathop{\sum}\limits_{j=1}^{d_r} y_j \\
& = & y_{\tau}-y_{\mathrm{id}}%
\end{array}%
\end{equation*}
Now using
\begin{equation*}
y_{\tau}-y_{\mathrm{id}}=z-y_{\mathrm{id}} \;\text{ in }\;{\mathbb{Q}}%
[K^*]/\langle y_{\tau}-z\rangle,
\end{equation*}
we have
\begin{equation}
{\mathrm{a}}\tilde{\xi}_{{\mathrm{div}}} |_{\tau}=z- y_{\mathrm{id.}}
\end{equation}
\subsection{Weighted Schubert Polynomials}

To define the weighted Schubert polynomials, we recall the definition of
double (equivariant) Schubert polynomial. To give the description, we recall
the following definition and notations which can also be found in \cite%
{manivel2001symmetric}.

\begin{defi}
Let $x_i$, $i=1,\ldots,n$ and $b_j$, $j\in \mathbb{N}$ be indeterminates.
Let ${\mathbb{Q}}[x]:={\mathbb{Q}}[x_1,\ldots, x_n]$ be the polynomial ring
over ${\mathbb{Q}}$. Let ${\mathbb{Q}}[x,b]:={\mathbb{Q}}[x]\otimes_{{%
\mathbb{Q}}}{\mathbb{Q}}[b]$. For each $\sigma \in S_{\infty}:=\cup_{k\geq
1}S_k$, the double (equivariant) Schubert polynomials $\mathfrak{S}_{\sigma}$
are defined as follows:
\begin{equation}
\mathfrak{S}_{\sigma}(x,b):= \mathop{\sum_{\si=\tau^{-1}\mu}}%
\limits_{l(\sigma)=l(\tau)+l(\mu)}\mathfrak{S}_{\tau}(x)\mathfrak{S}%
_{\mu}(-b)
\end{equation}
where $\mathfrak{S}_{\tau} $ is the usual Schubert polynomial.
\end{defi}

The usual and double Schubert polynomials have the following well-known
properties:

\begin{enumerate}
\item $\mathfrak{S}_{\sigma_0}(x)=x_1^{n-1}x_2^{n-2}\ldots x_{n-1}$ and $%
\mathfrak{S}_{s_d}(x)=x_1+x_2+\dotsb +x_d$.

\item $\mathfrak{S}_{\sigma_0}(x,b)=\prod_{i+j\leq n}(x_i-b_j)$, where $%
\sigma_0$ is the unique maximal element in $S_n$.

\item $\mathfrak{S}_{\mathrm{id}}(x)=\mathfrak{S}_{\mathrm{id}}(x,b)=1.$
\end{enumerate}

It is well known that the set $\{\mathfrak{S}_{\sigma}\}_{\sigma \in
S_{\infty}}$ forms a ${\mathbb{Q}}[b]$-basis of ${\mathbb{Q}}[x,b]$. We can
regard $H^*_T{\Sigma}$ as a ${\mathbb{Q}}[b]$-module by using the projection
map ${\mathbb{Q}}[b]\rightarrow H^*(BT)={\mathbb{Q}}[y_1,\ldots, y_n]$ that
sends $b_j$ to $y_j$, for $j=1,\ldots,n$ and to $0$ for $j>n$. Thus we have
a surjective homomorphism
\begin{equation*}
\theta: {\mathbb{Q}}[x,b]\rightarrow H^*_T(\Sigma)
\end{equation*}
of algebras over ${\mathbb{Q}}[b]$ which maps the polynomial $\mathfrak{S}%
_{\sigma}$ to the Schubert class $\tilde{\xi}_{{\sigma}^{-1}}$ of
codimension $l(\sigma)$, if $\sigma \in W^P$ and to zero otherwise (see \cite%
{kaji2010schubert}).

\begin{teo}
\label{theta}  The ring homomorphisms
\begin{equation*}
\theta_{{\mathrm{a}}}: {\mathbb{Q}}[x,b]\rightarrow H^*_K({\mathrm{a}}%
\Sigma^{\times}) \mathrm{\ \ \ \ and\ \ \ } \theta_{{\mathrm{w}}}: {\mathbb{Q%
}}[x,b]\rightarrow H^*_{T_{{\mathrm{w}}}}({\mathrm{w}}\Sigma)
\end{equation*}
are surjective,  where $\theta_{{\mathrm{a}}}:=\Pi^*\circ \theta$ and $%
\theta_{{\mathrm{w}}}:=(\Pi_{{\mathrm{w}}}^*)^{-1}\circ \theta_{{\mathrm{a}}}
$. Here $\mathfrak{S}_{\sigma}$ is mapped to ${\mathrm{a}}\tilde{\xi}%
_{\sigma^{-1}}$ by $\theta_{\mathrm{a}}$ and to ${\mathrm{w}}\tilde{\xi}%
_{\sigma^{-1}}$ by $\theta_{\mathrm{w}}$, if $\sigma\in W^P$ and to zero
otherwise.
\end{teo}

\begin{proof}
By composing the above map $\theta$ with isomorphisms in the commutative
diagram \eqref{CommuDiag}, we get the required result.
\end{proof}

By definition of double Schubert polynomial, we have
\begin{equation*}
\mathfrak{S}_{s_d}(x,b)=\mathfrak{S}_{s_d}(x)\mathfrak{S}_{\mathrm{id}}(-b)+%
\mathfrak{S}_{\mathrm{id}}(x)\mathfrak{S}_{s_d}(-b).
\end{equation*}
Then by the standard properties of the Schubert polynomial $\mathfrak{S}(x)$%
,
\begin{equation*}
\mathfrak{S}_{s_d}(x,b)=x_1+\dotsb+x_d-(b_1+\dotsb+b_d).
\end{equation*}
We now define the double Schubert polynomial corresponding to
\(\tilde{\xi}_{{\mathrm{div}}}\) as follows.
\begin{equation*}
\mathfrak{S}_{{\mathrm{div}}}(x,b):=\sum_{i=1}^{r}\mathfrak{S}%
_{s_{d_i}}(x,b).
\end{equation*}
So we have $\mathfrak{S}_{{\mathrm{div}}}(x,b)=x_{\mathrm{id}}-b_{\mathrm{id}%
}$, where\newline
$x_{\mathrm{id}}:=r \left(x_{1}+\cdots + x_{d_1}\right)
+(r-1)\left(x_{d_1+1}+\cdots+x_{d_2}\right)+\cdots+\left(x_{d_{r-1}+1}+%
\cdots+ x_{d_r}\right).$

\begin{coro}
\label{homo} The map
\begin{equation*}
\theta_{{\mathrm{a}}}: {\mathbb{Q}}[x,b]\rightarrow H^*_K({\mathrm{a}}%
\Sigma^{\times})
\end{equation*}
is a homomorphism of algebras over ${\mathbb{Q}}[x_{\mathrm{id}},b_1,b_2,
\ldots]$.
\end{coro}

\begin{proof}
As a consequence of the isomorphism $\Pi^*: H^*_T(\Sigma)\rightarrow H^*_K({%
\mathrm{a}}\Sigma^{\times})$ and the definition \eqref{div} of $\tilde{\xi}_{{\mathrm{div%
}}}$, we have
\begin{equation*}
(\Pi^*)^{-1}(z)=\tilde{\xi}_{{\mathrm{div}}}+y_{\mathrm{id}}.
\end{equation*}
This and the expression for $\mathfrak{S}_{{\mathrm{div}}}(x,b)$ give
\begin{equation*}
\theta_{{\mathrm{a}}}(x_{\mathrm{id}})={\mathrm{a}}\tilde{\xi}_{{\mathrm{div}%
}}+y_{\mathrm{id}}=z.
\end{equation*}
By Theorem \ref{theta}, $\theta_{{\mathrm{a}}}$ is a homomorphism of
algebras over ${\mathbb{Q}}[b]$. Thus we have extended the ring of
coefficients to ${\mathbb{Q}}[x_{\mathrm{id}},b_1,b_2, \ldots]$ by mapping $%
x_{\mathrm{id}}$ to $z$, $b_i$ onto $y_i$ for $1\leq i\leq n$ and $b_i$ to
zero for $i>n$.
\end{proof}

Let $w_l\in {\mathbb{Z}}_{\geq 0}$, for $l\in {\mathbb{N}}$ be such that the
first $n$ terms $w_1,\ldots,w_n$ are the same as weights chosen to define
the weighted flag variety ${\mathrm{w}}\Sigma$. Let
\begin{equation*}
b^{\mathrm{w}}_l:= b_l+\frac{w_l}{u}x_{\mathrm{id}},\ \ \ l\in {\mathbb{N}}.
\end{equation*}
As $\{b_l^{\mathrm{w}}\ : \ l\in {\mathbb{N}} \}$ is a set of algebraically
independent variables so ${\mathbb{Q}}[b^{\mathrm{w}}]:={\mathbb{Q}}[b_1^{%
\mathrm{w}},b_2^{\mathrm{w}},\ldots]$ is a polynomial ring. Moreover, we
have a canonical isomorphism of rings
\begin{equation*}
{\mathbb{Q}}[x,b]\cong {\mathbb{Q}}[x,b^{\mathrm{w}}]
\end{equation*}
via
\begin{equation}  \label{isomor}
x_i \mapsto x_i \text{ and } b_l\mapsto b_l^{\mathrm{w}}-(w_l/u)x_{\mathrm{id%
}}.
\end{equation}%
The following result follows from Corollary \ref{homo}.

\begin{prop}
\label{prop:div}
The map $\theta_{w}: {\mathbb{Q}}[x,b^{\mathrm{w}}]\rightarrow H_T^*({%
\mathrm{w}}\Sigma)$ is a surjective homomorphism of algebras over ${\mathbb{Q%
}}[b^{\mathrm{w}}]$ via the following projection
\begin{equation*}
{\mathbb{Q}}[b^{\mathrm{w}}]\rightarrow {\mathbb{Q}}[y_1^{\mathrm{w}}%
,\ldots,y_n^{\mathrm{w}}],
\end{equation*}
$b_l^{\mathrm{w}}\mapsto y_l^{\mathrm{w}}$ for $l=1,\ldots, n$ and $b_l^{%
\mathrm{w}}\mapsto 0$, otherwise.
\end{prop}

Using the isomorphism of algebras ${\mathbb{Q}}[x,b]\cong {\mathbb{Q}}[x,b^{%
\mathrm{w}}]$ defined by \eqref{isomor}, we are now able to define weighted
Schubert polynomials ${\mathrm{w}} \mathfrak{S}_{\sigma}(x,b^{\mathrm{w}})$.

\begin{defi}
Weighted Schubert polynomials ${\mathrm{w}} \mathfrak{S}_{\sigma}(x,b^{%
\mathrm{w}})$ are defined as images of double Schubert polynomials $%
\mathfrak{S}_{\sigma}(x,b)$ under the isomorphism ${\mathbb{Q}}[x,b]\cong {%
\mathbb{Q}}[x,b^{\mathrm{w}}]$. That is,
\begin{equation*}
{\mathrm{w}} \mathfrak{S}_{\sigma}(x,b^{\mathrm{w}}):= \mathop{\sum_{\si=%
\tau^{-1}\mu}}\limits_{l(\sigma)=l(\tau)+l(\mu)}\mathfrak{S}_{\tau}(x)%
\mathfrak{S}_{\mu}(b_1^{\mathrm{w}}-(w_1/u)x_{\mathrm{id}},b_2^{\mathrm{w}}%
-(w_2/u)x_{\mathrm{id}},\ldots).
\end{equation*}
\end{defi}

We give a description of weighted Schubert classes in terms of weighted
Schubert polynomials. We will use the  following  result of  Kaji on double Schubert polynomials.

\begin{prop}
(\cite[Prop. 4.9]{kaji2011equivariant}) For $\sigma, \tau$ in $W^P$, we have

\begin{equation}  \label{eqSchClass}
\mathfrak{S}_{\sigma}(b_{\tau(1)},\ldots, b_{\tau(n)};b)=\Bigg\{
\begin{array}{cc}
\mathop{\prod}\limits_{(i,j)\in {\mathrm{Inv}}_P(\sigma)}(b_{j}-b_i) &
\tau=\sigma \\
0 & \tau \nsucceq \sigma%
\end{array}%
.
\end{equation}
In particular, for $b_i=y_i, \forall i\in {\mathbb{N}}$, where $y_i=0$ for
all $i>n$, we have
\begin{equation*}
\mathfrak{S}_{\sigma}(b_{\tau(1)},\ldots, b_{\tau(n)};b)= \tilde{\xi}%
_{\sigma}|_{\tau}.
\end{equation*}
\end{prop}

Consider the map $\alpha_{\tau}:{\mathbb{Q}}[x,b^{\mathrm{w}}]\rightarrow {%
\mathbb{Q}}[b^{\mathrm{w}}]$ defined as
\begin{equation*}
x_i \mapsto b^{\mathrm{w}}_{\tau_i}-(w_{\tau_i}/w_{\tau})b^{\mathrm{w}}%
_{\tau}\;\; \forall \; i=1,\ldots,n,
\end{equation*}
where $\tau_i=\tau(i).$ The image of $x_{\mathrm{id}}$ under the map $%
\alpha_{\tau}$ is $(u/w_{\tau})b_{\tau}^{\mathrm{w}}$. Thus we have:

\begin{equation}  \label{lemma}
\alpha_{\tau}({\mathrm{w}} \mathfrak{S}_{\sigma}(x,b^{\mathrm{w}}))=%
\mathfrak{S}_{\sigma}(d^{\mathrm{w}}_{\tau_1},\dots,d^{\mathrm{w}}%
_{\tau_n},d^{\mathrm{w}}_1,d^{\mathrm{w}}_2,\ldots)
\end{equation}
for $\sigma,\tau\in W^P$, where $d^{\mathrm{w}}_{\tau_i}=b^{\mathrm{w}}%
_{\tau_i}-(w_{\tau_i}/w_{\tau})b^{\mathrm{w}}_{\tau}.$

\begin{lema}
For $\sigma, \tau$ in $W^P$, we have

\begin{equation*}
{\mathrm{w}} \tilde{\xi}_{\sigma}|\tau=\alpha_{\tau}({\mathrm{w}} \mathfrak{S%
}_{\sigma}(x,b^{\mathrm{w}}))|_{b_i^{\mathrm{w}}=y_i^{\mathrm{w}}}.
\end{equation*}
\end{lema}

\begin{proof}
The previous Proposition and (\ref{lemma}) give us the desired expression
for ${\mathrm{w}} \tilde{\xi}_{\sigma}|\tau$.
\end{proof}

 We  define the weighted  Schubert polynomial ${\mathrm{w}}
\mathfrak{S}_{\sigma}(x)$ as
\begin{equation*}
{\mathrm{w}} \mathfrak{S}_{\sigma}(x,0)=\mathfrak{S}_{\sigma}(x,-(w_1/u)x_{
\mathrm{id}},\ldots).
\end{equation*}

\begin{prop}
The polynomials ${\mathrm{w}}\mathfrak{S}_{\sigma }(x)$, $\sigma \in W^{P}$
form a ${\mathbb{Q}}$-basis of ${\mathbb{Q}}[x_{1},\dotsc ,x_{n}]^{W_P}/I,$ where $%
I$ is the ideal generated by $$f(x_{1},\ldots ,x_{r})-f(\frac{-w_{i}}{u}x_{\mathrm{id}%
},\ldots ,\frac{-w_{r}}{u}x_{\mathrm{id}})$$ \ for all
$W$-invariant polynomials $f$ of positive degree. Moreover,  there is
a surjective ring homomorphism
\begin{equation*}
\dfrac{\mathbb{Q}[x_{1},\ldots,x_{n}]^{W_P}}{I}\rightarrow H^{\ast }({\mathrm{w}}\Sigma )
\end{equation*}
defined as
\begin{equation*}
{\mathrm{w}}\mathfrak{S}_{\sigma }(x)\mapsto {\mathrm{w}}\xi _{\sigma ^{-1}}.
\end{equation*}
\end{prop}
\proof  By Borel Construction \cite[p.24]{guillemin2006gkm}, it is well-known that
$$H^*_T(\Si)\cong \dfrac{{\mathbb{Q}[x_{1},\dotsc ,x_{n}]}^{W_P}\otimes \mathbb{Q}[b_{1},\dotsc ,b_{n}] }{I},$$
 where $
I$ is the ideal generated by $f(x_{1},\ldots ,x_{r})-f(b_1,\dots,b_r)$   for all
$W$-invariant polynomials $f$ of positive degree.
Also, the double Schubert polynomials $\{\mathfrak{S}_{\sigma }(x,b)\  :\ \si\in W_P\}$ form a basis for  $$\dfrac{{\mathbb{Q}[x_{1},\dotsc ,x_{n}]}^{W_P}\otimes \mathbb{Q}[b_{1},\dotsc ,b_{n}] }{I}.$$
Thus the result follows from the definition of weighted Schubert polynomial.

\begin{teo}
\label{Th:Chev-Mon}
For any simple reflection $s_{d}\in W^P$ and $\sigma\in W^P$, the
Chevalley--Monk's formula for weighted double Schubert polynomials is given
by
\begin{equation*}
{\mathrm{w}}\mathfrak{S}_{s_d}(x){\mathrm{w}}\mathfrak{S}_{\sigma}(x)= \frac{
\sum_{i=1}^{d}(w_{i}-w_{\sigma_i})}{u}x_{\mathrm{id}}{\mathrm{w}}\mathfrak{S}
_{\sigma}(x)+\mathop{\sum_{i\leq d< j}}\limits_{l((ij)\sigma)=l(\sigma)+1}{
\mathrm{w}}\mathfrak{S}_{(ij)\sigma}(x).
\end{equation*}
\end{teo}

\begin{proof}
From\cite{manivel2001symmetric}, we have a nice formula for the product of
double Schubert polynomials

\begin{equation*}
\mathfrak{S}_{s_d}(x,b)\mathfrak{S}_{\sigma}(x,b)=(b_{\sigma_1}+\cdots+b_{
\sigma_{d}}-b_1-\cdots-b_d)\mathfrak{S}_{\sigma}(x,b)+\mathop{\sum_{i\leq d<
j}}\limits_{l((ij)\sigma)=l(\sigma)+1}\mathfrak{S}_{(ij)\sigma}(x,b)
\end{equation*}
Replacing $b_i$ by $-\frac{w_i}{u}x_{\mathrm{id}}$ in above equation, we get
the result.
\end{proof}


As a concluding remark, it is opportune to mention that the geometric and
topological aspects (including Schubert calculus) of weighted flag
varieties, as a subject is out there to be explored. We expect this to grow
into an area with potential research directions, for instance \cite{ANQ}.
The next obvious question seems to be of computing cohomology of complete
intersections in weighted flag orbifolds. Similarly we hope that results
from equivariant $K$-theory and equivariant quantum cohomology of
homogeneous spaces could also be generalized to these spaces.

\subsection*{Acknowledgement}

We wish to thank Waqar Ali Shah for several helpful discussions. We also
thank Frank Sottile and Bal{\'a}zs Szendr{\H o}i for their comments on
earlier drafts of this article. Thanks are also due to Shizu Kaji, Allen
Knutson and Loring Tu for some useful conversations. Last but not least, we are indebted to the anonymous referee
for his/her comments and suggestions which significantly improved the earlier version of this paper.
HA and MIQ were supported by the HEC's NRPU research grant ``5906/Punjab/NRPU/HEC/2016".
MIQ was on a fellowship of Alexander--Von--Humboldt foundation during a part
of this paper. 
\bibliographystyle{amsalpha}
\bibliography{References}

\end{document}